\documentclass[12pt,reqno]{amsart}
\usepackage[top=1in, bottom=1in, left=1in, right=1in]{geometry}
\usepackage{times, amsthm, amssymb, amsmath, amsfonts, amsthm, bm, graphicx, mathrsfs}
\usepackage{latexsym,wasysym}
\usepackage{amscd}
\usepackage{float}
\usepackage{enumerate}
\usepackage[all]{xy}
\usepackage[usenames,dvipsnames]{color}
\usepackage{tikz}
\usepackage[pagebackref=true]{hyperref}
\usepackage{lineno}

 \numberwithin{equation}{section}

\usepackage{algpseudocode}

\usepackage{mathtools}

\usepackage{xy}
\input xy
\xyoption{all}

\newcommand{\excise}[1]{}%{$\star$\textsc{#1}$\star$}

\newtheorem{theorem}{Theorem}[section]
\newtheorem{proposition}[theorem]{Proposition}
\newtheorem{lemma}[theorem]{Lemma}

\newtheorem*{theorem*}{Theorem}

\theoremstyle{definition}
\newtheorem{definition}[theorem]{Definition}
\newtheorem{example}[theorem]{Example}

\newtheorem{notation}[theorem]{Notation}

\numberwithin{equation}{section}

\newcommand\<{\langle}
\newcommand\CC{{\mathbb C}}

\newcommand\NN{{\mathbb N}}
\newcommand\QQ{{\mathbb Q}}
\newcommand\RR{{\mathbb R}}

\newcommand\ZZ{{\mathbb Z}}

\newcommand\cI{{\mathcal I}}
\newcommand\cJ{{\mathcal J}}

\newcommand\cO{{\mathcal O}}

\newcommand\sF{{\mathscr F}}

\newcommand\del{\partial}

\newcommand\nsupp{{\rm nsupp}}
\newcommand\psupp{{\rm psupp}}

\renewcommand\div{{\rm div}}

\renewcommand\>{\rangle}

\newcommand\0{\mathbf{0}}

\newcommand\kk{\Bbbk}

\renewcommand\AA{{\mathbb A}}

\newcommand\minus{\smallsetminus}

\DeclareMathOperator\Hom{Hom} % Hom
\DeclareMathOperator\Spec{Spec} %
\DeclareMathOperator\relint{int} %relative interior
\DeclareMathOperator\cvx{conv} %convex hull

\begin{document}%%%%%%%%%%%%%%%%%%%%%%%%%%%%%%%%%%%%%%%%%%%%%%%%%%%%%%%%
%%%%%%%%%%%%%%%%%%%%%%%%%%%%%%%%%%%%%%%%%%%%%%%%%%%%%%%%%%%%%%%%%%%%%%%%

\mbox{}
\title{Bernstein--Sato polynomials on normal toric varieties}
\author{Jen-Chieh Hsiao}
\address{Mathematics Department\\ Cheng Kung University\\ Tainan City 70101, Taiwan}
\email{jhsiao@mail.ncku.edu.tw}
\thanks{JCH was partially supported by MOST grant 105-2115-M-006-015-MY2.}
\author{Laura Felicia Matusevich}
\address{Mathematics Department\\Texas A\&M University\\College Station, TX 77843}
\email{laura@math.tamu.edu}
\thanks{LFM was partially supported by NSF grant DMS 1500832.}


\begin{abstract}
We generalize the Bernstein--Sato polynomials of Budur,
Musta{\c{t}}{\v{a}}, and Saito
to ideals in normal semigroup rings.
In the case of monomial ideals, we also
relate the roots of the Bernstein--Sato polynomial to the jumping
coefficients of the corresponding multiplier ideals. In order to prove
the latter result, we obtain a new combinatorial description for the
multiplier ideals of a monomial ideal in a normal semigroup ring.
\end{abstract}
\maketitle

%%%%%%%%%%%%%%%%%%%%%%%%%%%%%%%%%%%%%%%%%%%%%%%%%%%%%%%%%%%%%%%%%%%%%%%%%
\section{Introduction}\label{s:intro}%%%%%%%%%%%%%%%%%%%%%%%%%%%%%%%%%%%%
%%%%%%%%%%%%%%%%%%%%%%%%%%%%%%%%%%%%%%%%%%%%%%%%%%%%%%%%%%%%%%%%%%%%%%%%%

Let $f \in \CC[x]=\CC[x_1, \dots, x_n]$ be a non-constant polynomial
and let $D =\CC[x,\del_x]$ be the Weyl algebra. The Bernstein--Sato
polynomial (or $b$-function) of $f$, introduced independently
in~\cite{Ber72} and~\cite{SS72}, is the monic polynomial $b_f(s) \in 
\CC[s]$ of smallest degree such that there exists $P(s) \in
D[s]=D\otimes_\CC \CC[s]$ satisfying the functional equation 
$P(s)\cdot f^{s+1} = b_f(s) f^s$. It is well known that $b_f(s)=s+1$
if and only if the hypersurface defined by $f$ is
nonsingular. Moreover, the roots of $b_f(s)$ are related to the
eigenvalues of the Milnor monodromy of $f$~\cite{Mal83}, and the poles of
the local zeta function associated to $f$~\cite{Igu00}.  

Another important singularity invariant are the multiplier ideals
$\mathcal J(\CC^n, \alpha\langle f\rangle)$ associated to the
hypersurface $f$ in $\CC^n$~\cite{Laz04b}. They can be defined via an
embedded log resolution of the pair $(\CC^n, \langle f\rangle)$. When
the coefficient $\alpha$ varies, the multiplier ideal $\mathcal
J(\CC^n, \alpha \langle f \rangle)$ jumps.  In fact, the smallest
jumping coefficient of $(\CC^n, \langle f \rangle)$ (i.e. the
log-canonical threshold of the hypersurface defined by $f$) is the
smallest root of $b_f(-s)$~\cite{Yan83,Lic89,Kol97}. Generalizing this
result, one of the main theorems in~\cite{ELSV} states that if $\xi$
is a jumping coefficient in $(0,1]$ of the pair $(\CC^n, \langle
f\rangle)$, then it is a root of $b_f(-s)$. Along the same lines,
Budur, Musta{\c{t}}{\v{a}}, and Saito  extended the notion of
$b$-functions to the case of arbitrary ideals on smooth affine
varieties~\cite{BMS06}.  Utilizing the theory of $V$-filtrations
from~\cite{Kas83} and~\cite{Mal83}, they showed that the $b$-function
of an ideal is independent of the choice of generators. Furthermore,
they generalized the connection established in~\cite{ELSV} between the
roots of their $b$-functions and the jumping coefficients of
multiplier ideals  to this more general setting. 
 
\begin{theorem}[\cite{BMS06}]\label{thm:BMSmainthm}
Let $I$ be an ideal on a smooth affine variety $X$ over $\CC$. Then
the log-canonical threshold of the pair $(X,I)$ coincides with the
smallest root $\alpha_I$ of the $b$-function $b_I(-s)$, and any
jumping coefficient of $(X,I)$ in $[\alpha_I, \alpha_I+1)$ is a root
of $b_I(-s)$.
\end{theorem}

As the theory of multiplier ideals  is generalized to the case of
arbitrary ideals on normal varieties~\cite{dFH09}, it is interesting
to explore possible generalizations of
Theorem~\ref{thm:BMSmainthm}. To this end, the first task is to seek
for a suitable candidate that plays the role of the Weyl algebra. In
the current work, we restrict to the case where the ambient variety
$X$ is an affine normal toric variety. There exists an explicit
combinatorial description of Grothendieck's ring of differential
operators $D_X$ on the toric variety $X$~\cite{Mus87,Jon94,ST01},
serving as an analog of Weyl algebra in the case of $X=\CC^n$. 
We apply this description of $D_X$ and the results
in~\cite{BMS06,BMS06b} to construct the $b$-function $b_I(s)$ for an
ideal $I$ on the toric variety $X$.  
To state our results more explicitly, let us set up some notations
that will be used throughout this paper. 

\begin{notation}
\label{not:intro} 
\begin{enumerate}
\item \label{not:intro:cone}
Throughout this article,
we work over an algebraically closed field $\kk$ of
characteristic $0$ instead of the complex number field $\CC$. Our
results are valid in this setting. 
The coordinate ring of the affine normal toric variety $X$ is
represented by the semigroup ring $\kk[\NN A]$ generated by the
columns $a_1, \dots, a_d$ of a rank $d$ matrix $A \in \ZZ^{d \times
  m}$. We assume that $\ZZ A = \ZZ^d$, the cone $C=\RR_{\ge 0}A$ in
$\RR^d$ over $A$ is strongly convex, and the semigroup $\NN A$ is
normal, meaning that $C \cap \ZZ^d = \NN A$.
The faces of $C$ are denoted by greek letters $\sigma, \tau$ etc. This
may refer to an index set, to a collection of columns of $A$, or to
the actual face of the cone.
If $\sigma$ is a \emph{facet} of $C$, define its \emph{primitive
  integral support function} $F_\sigma : \RR^d \to \RR$ such that 
\begin{enumerate}
\item $F_\sigma(\ZZ A) = \ZZ$,
\item $F_\sigma(a_i) \geq 0$ for $i=1,\dots,n$,
\item $F_\sigma(a_i)=0$ for $a_i\in \sigma$.
\end{enumerate}
Define the linear map $F: \RR^d \to \RR^{\mathscr{F}} \text{ by } F(p) =
(F_{\sigma}(p))_{\sigma \in \mathscr{F}}$, where $\sF$ is the collection of facets of $C$.
\item We consider $\kk[\NN A]$ as a subring of the Laurent polynomial ring
$\kk[y_1^{\pm 1},\dots,y_d^{\pm 1}]$. Then the ring of differential
operators $D_A$ on the toric variety $X$ can be represented as a
subring of the 
linear partial differential operators on $\del_{y_1},\dots,\del_{y_d}$
with Laurent polynomial coefficients,
\[
D_A = \bigoplus_{u \in \ZZ^d} y^u \{ f(\theta) \in
\kk[\theta_1,\dots,\theta_d] \mid f \text{ vanishes on } \NN A \minus
(-u+\NN A) \}, 
\]
where $\theta_1,\dots,\theta_d$ are the commuting operators 
$y_1\del_{y_1},\dots,y_d\del_{y_d}$.
\item The linear map $F:\RR^d \rightarrow \RR^\sF$ induces a ring homomorphism
\[ \kk[\NN A] \rightarrow \kk[\NN^\sF]
\qquad \text{ via } \qquad
\sum_{j=1}^k \lambda_{j} y^{\beta_j} \mapsto \sum_{j=1}^k \lambda_j x^{F(\beta_j)},
\]
where $\beta_1,\dots,\beta_j \in \NN A$ and $(x_\sigma)_{\sigma \in
  \sF}$ denote algebraically independent variables, so that
$\kk[\NN^{\sF}] =
\kk[x_\sigma \mid \sigma \in \sF]$.
For an ideal $I$ in $\kk[\NN A] \subset \kk[y_1^{\pm 1},\dots,y_d^{\pm
  1}]$, we abuse notation to denote 
\[
F(I) = \left \{ \sum \lambda_jx^{F(\beta_j)} \in \kk [\NN ^\sF ] \,
  \Big| \, \sum_j \lambda_j y^{\beta_j} \in I \right\}.
\]
Note that $F(I)$ is an ideal in the semigroup ring $\kk[F(\NN A)]$.
The ideal in $\kk[\NN^\sF]$ generated by $F(I)$ is denoted by $ J=\< F(I)\>$.
\end{enumerate}
\end{notation}

The first goal of this paper is to extend the notion of $b$-functions
to the toric setting. We use the same definition as in~\cite{BMS06}
except that the Weyl algebra $D$ is replaced by the ring of
differential operators $D_A$ on $X$. Analyzing the combinatorics of
the toric algebras, we obtain the following result. 
\begin{theorem*}[Theorem \ref{thm:ToricbFunctionGeneral}]
The $b$-function $b_I(s)$ of an ideal $I$ on the toric variety $X$
coincides with the $b$-function $b_J(s)$ of the ideal $J = \<F(I)\>$
in the polynomial ring $\kk[\NN^\sF]$.
\end{theorem*} 
 In particular, avoiding the study of the $V$-filtration in the toric
 case, we can still conclude that the toric $b$-function $b_I(s)$ is
 independent of the choices of generators of $I$. 

To generalize Theorem~\ref{thm:BMSmainthm} to the toric setting, we
further restrict to the case where $I$ is a monomial ideal (i.e. an ideal
defining a torus-invariant subvariety of $X$). In this case, the
multiplier ideals of the pair $(X, I)$ were described in~\cite{Bli04},
generalizing the work in~\cite{How01} on the multiplier ideals of
monomial ideals in polynomial algebras. Using these formulas, we
obtain a new expression of the multiplier ideals of monomial ideals on
toric varieties using the map $F$
(Proposition~\ref{prop:ToricHowaldFormulaGeneral}), 
\[
\cJ(X,\alpha I) = \< y^v \in \kk[\NN A] \mid F(v)+e  \in
\relint(\alpha P_{J})\>,
\]
where $e \in \NN^\sF$ is such that $e_\sigma=1$ for all $\sigma \in
\sF$ and $\relint (P_J)$ is the relative interior of the Newton
polyhedron $P_J$ of the monomial ideal $J$ in $\kk[\NN^\sF]$. 
In particular, we obtain the identity,
\[ 
\mathcal J(X,\alpha I) = \kk[\NN A] \cap \mathcal J(\kk^\sF, \alpha
J),
\]
which relates the jumping coefficients of the pair $( X, I)$  to those
of $(\kk^\sF, J)$. Since the jumping coefficients of $(\kk^\sF, J)$
are related to the roots of $b_J(-s)$ by Theorem~\ref{thm:BMSmainthm},
our arguments give the following theorem. 
\begin{theorem*}[Theorem~\ref{thm:ToricRootsJumpingCoefficientsCorrespondence}]
Let $I$ be a monomial ideal on an affine normal toric variety $X$ over
an algebraically closed field $\kk$ of characteristic $0$. Then the
log-canonical threshold of the pair $(X,I)$ coincides with the
smallest root $\alpha_I$ of the $b$-function $b_I(-s)$, and any
jumping coefficients of $(X,I)$ in $[\alpha_I, \alpha_I+1)$ are roots
of $b_I(-s)$. 
\end{theorem*}

%%%%%%%%%%%%%%%%%%%%%%%%%%%%%%%%%%%%%%%%%%%%%%%%%%%%%%%%%%%%%%%%%%%%%%%%%
\section{Bernstein--Sato Polynomials of Budur, Musta\c{t}\u{a}, and Saito}\label{s:bFunctions}%%%%%%%%%%%%%%
%raggedbottom%%%%%%%%%%%%%%%%%%%%%%%%%%%%%%%%%%%%%%%%%%%%%%%%%%%%%%%%%%%%

We recall relevant background on Bernstein--Sato polynomials for
ideals from the work of Budur, Musta\c{t}\u{a} and Saito
\cite{BMS06}. Here we concentrate on the 
case of subvarieties of affine spaces, but the definition works for
subvarieties of a smooth variety.

Let $I$ be an ideal in $\kk[x_1,\dots,x_n]$, generated by
$f_1,\dots,f_r$.  Denote by $D$ the Weyl algebra on $x_1,\dots,x_n$. 
If $s_1,\dots,s_r$ are indeterminates, consider
\[
\kk[x_1,\dots,x_n][\prod_{i=1}^rf_i^{-1},s_1,\dots,s_r] \prod_{i=1}^r
f_i^{s_i} .
\]
This is a $D[s_{ij}]$-module, where $s_{ij} = s_it_i^{-1}t_j$, and the
action of the operator $t_i$ is given by $t_i(s_j) = s_j + \delta_{ij}$ (the
Kronecker delta). 

\begin{definition}
\label{def:bFunction}
The \emph{Bernstein--Sato polynomial} or \emph{$b$-function}
associated to $f=(f_1,\dots,f_r)$ is defined to be the monic polynomial of
the lowest degree in $s=\sum_{i=1}^r s_i$ satisfying a relation of the
form
\begin{equation}
\label{eqn:bFunctionDef}
b_f(s) \prod_{i=1}^r f_i^{s_i} = \sum_{k=1}^r P_kt_k \prod_{i=1}^r f_i^{s_i},
\end{equation}
where $P_1,\dots,P_r \in D[s_{ij}\mid i,j \in \{1,\dots,r\}]$. 
\end{definition}

In~\cite[Section~2]{BMS06}, Budur, Musta\c{t}\u{a} and Saito show that:

\begin{enumerate}
\item $b_f(s)$ is a nonzero polynomial,
\item $b_f(s)$ is independent of the generating set of $I$.
\end{enumerate}

\subsection{An alternative way to define the Bernstein--Sato
  polynomial}
\label{ssec:bFunctionDef2}

In~\cite[Section~2.10]{BMS06} 
Budur, Musta\c{t}\u{a} and Saito give the following way to compute the
$b$-function.

For $c=(c_1,\dots,c_r) \in \ZZ^r$, define $\nsupp(c) = \{i \mid
c_i<0\}$. The Bernstein--Sato polynomial $b_f(s)$ is the monic
polynomial of the smallest degree such that $b_f(s)\prod_{i=1}^r
f_i^{s_i}$ belongs to the $D[s_1,\dots,s_r]$-submodule generated by
\begin{equation}
\label{eqn:bFunctionDef2}
\prod_{i \in \nsupp(c)} \binom{s_i}{-c_i} \cdot \prod_{i=1}^r f_i^{s_i+c_i} 
\end{equation}
where $c=(c_1,\dots,c_r)$ runs over the elements of $\ZZ^r$ such that
$\sum_{i=1}^r c_i=1$. Here $s=\sum_{i=1}^rs_i$ and $\binom{s_i}{m} =
s_i(s_i-1)\cdots(s_i-m+1)/m!$. 

The advantage of this approach is that one now works in the ring
$D[s_1,\dots,s_r]$, and need no longer consider the operators $s_{ij}$ or $t_k$.

\subsection{A monomial-ideal-specific way to define the
  Bernstein--Sato polynomial}
\label{ssec:bFunctionDefMonomial}

The definition given in the previous subsection can be further
specialized to the case of monomial ideals.

Let $\mathfrak{a}$ be a monomial ideal in $\kk[x_1,\dots,x_n]$ whose minimal
monomial generators are $x^{\alpha_1},\dots,x^{\alpha_r}$. Set
$\ell_\alpha(s_1,\dots,s_r) = s_1 \alpha_1+\cdots +s_r \alpha_r$, and 
$\ell_i(s)$ to be the $i$th coordinate of this vector.
For $v \in \ZZ^r$, let $\psupp(v) = \{ i \mid v_i > 0\}$. 

For $c=(c_1,\dots,c_r) \in \ZZ^r$ such that $\sum_{i=1}^r c_i = 1$,
set
\[
g_c(s_1,\dots,s_r) = \prod_{j \in \nsupp(c)} \binom{s_j}{-c_j} \cdot
\prod_{i \in \psupp(\ell_\alpha(c))} \binom{\ell_i(s_1,\dots,s_r)+\ell_i(c)}{\ell_i(c)},
\]
and define $I_\mathfrak{a}$ to be the ideal of $\QQ[s_1,\dots,s_r]$
generated by the polynomials $g_c$. 

Then~\cite[Proposition~4.2]{BMS06} asserts that $b_f(s)$ is the monic
polynomial of smallest degree such that $b_f(\sum_{i=1}^r s_i)$
belongs to the ideal $I_\mathfrak{a}$.

%%%%%%%%%%%%%%%%%%%%%%%%%%%%%%%%%%%%%%%%%%%%%%%%%%%%%%%%%%%%%%%%%%%%%%%%%
\section{Bernstein--Sato Polynomials on Toric Varieties}
\label{s:bFunctionSemigroupRing}
%%%%%%%%%%%%%%%%%%%%%%%%%%%%%%%%%%%%%%%%%%%%%%%%%%%%%%%%%%%%%%%%%%%%%%%%%
The goal of this section is to extend Definition~\ref{def:bFunction}
to the toric setting. We first recall some background on the ring of
differential operators of a toric variety.

%%%%%%
\subsection{Rings of differential operators on toric varieties}
\label{ss:ToricDifferentialOperators}
%%%%%%
For a commutative algebra $R$ over a commutative ring $k$,
Grothendieck's ring of $k$-linear differential operators $D_k(R)$ of
$R$~\cite{EGA4} is the $R$-subalgebra of $\Hom_k(R,R)$ defined by  
\[
D_k(R) = \bigcup_{i \in \NN \cup \{0\}} D_i,
\] 
where $D_0=R$ and $D_i=\{ f \in \Hom_k(R,R) \mid fr-rf \in D_{i-1}
\text{ for all } r \in R\}$ for $i \ge 1$.  
When $R$ is the polynomial ring of $n$ variables over an algebraically
closed field $k$ of characteristic $0$, the ring $D_k(R)$ coincides
with the $n$th Weyl algebra. Moreover, it is well known that if $R$ is
regular over $k$, then $D_k(R)$ is the $R$-algebra generated by the
$k$-derivations of $R$. The most influential applications of
$D$-module theory, and in particular, the $b$-functions
of~\cite{BMS06}, are built in this setting. 

When $R$ is not regular, the ring $D_k(R)$ is not
well-behaved~\cite{BGG72}. Nonetheless, when $R=\kk[\NN A]$ is a toric
algebra over 
an algebraically closed field $\kk$ of characteristic $0$, there
exists an explicit combinatorial description of $D_\kk(\kk[\NN
A])$. This expression of $D_\kk(\kk[\NN A])$ was obtained
independently in~\cite{Mus87} and~\cite{Jon94} when $\NN A$ is normal,
and was further extended to the not necessarily normal case in~\cite{ST01}. 
We denote the ring of differential operators of the toric algebra
$\kk[\NN A]$ by
\[
D_A=D_\kk(\kk[\NN A]).
\]

One way to understand $D_A$ is to identify $\kk[\NN A]$ as a subring
of the Laurent polynomial ring $\kk[y_1^{\pm 1},\dots,y_d^{\pm 1}]$
and study the $\kk$-linear differential operators of $\kk[y_1^{\pm
  1},\dots,y_d^{\pm 1}]$. It is known that if $S$ is a
multiplicatively closed subset of a $k$-algebra $R$, then 
$D_k(S^{-1}R) \cong S^{-1}R \otimes_R D_k(R)$. It  follows that 
\[
D_\kk(\kk[y_1^{\pm 1},\dots,y_d^{\pm 1}]) = \kk[y_1^{\pm
  1},\dots,y_d^{\pm 1}]\< \del_{y_1}, \dots, \del_{y_d}\>,
\] 
is a localization of the $n$th Weyl algebra. Moreover, one can realize 
\[ 
D_A=\{\delta \in D_\kk(\kk[y_1^{\pm 1},\dots,y_d^{\pm 1}]) \mid
\delta (\kk[\NN A]) \subseteq \kk[\NN A] \}
\] 
as a subring of
$D_\kk(\kk[y_1^{\pm 1},\dots,y_d^{\pm 1}]) $. In particular, we have
the $\ZZ^d$-graded expression 
\begin{equation}
\label{eqn:DA}
D_A = \bigoplus_{u \in \ZZ^d} y^u \{ f(\theta) \in
\kk[\theta_1,\dots,\theta_d] \mid f \text{ vanishes on } \NN A \minus
(-u+\NN A) \}
\end{equation}
where $\theta_1,\dots,\theta_d$ are the commuting operators $y_i
\del_{y_i}$, for $i=1,\dots,d$.
Since we are only concerned with the case when $\NN A$ is normal, for
a given $u \in \ZZ^d$, the ideal  
\[
\{ f(\theta) \in
\kk[\theta_1,\dots,\theta_d] \mid f \text{ vanishes on } \NN A \minus
(-u+\NN A) \}
\] 
is the principal ideal in  $\CC[\theta_1,\dots,\theta_d]$ generated by 
\begin{equation}
\label{eqn:differentialGenerator}
\prod_{F_{\sigma}(u)>0} \prod_{j=0}^{F_{\sigma}(u)-1} ( F_\sigma(\theta_1,\dots,\theta_d)-
j ),
\end{equation}
where the product runs through all facets $\sigma$ (with
$F_\sigma(u)>0$) of the cone $C$ associated to $\NN A$, and $F_\sigma
$ is the primitive integral support function of $\sigma$ as defined in
Notation~\ref{not:intro}.\ref{not:intro:cone}.

We are now ready to define the $b$-function for an ideal $I$ in
$\kk[\NN A]$. To make the exposition more transparent, we first treat
the case when $I$ is an monomial ideal. 
%%%%%%
\subsection{Monomial ideals in semigroup rings}
\label{ss:toricBernsteinSatoMonomial}
%%%%%%

Let $\mathfrak{a}$ be a monomial ideal in $\kk[\NN A]$ whose minimal
monomial generators are Laurent monomials $y^{\beta_1},\dots,y^{\beta_r}$. 
For $f = (y^{\beta_1},\dots,y^{\beta_r})$, we define the
Bernstein--Sato polynomial $b_f(s)$ exactly as in
Definition~\ref{def:bFunction}, except that the Weyl algebra $D$ is
replaced by the ring of differential operators $D_A$ on $\kk[\NN A]$.

We note that the reduction in Subsection~\ref{ssec:bFunctionDef2} goes
unchanged when we switch to the setting of semigroup rings, since that
reduction is concerned only with the operators $s_{ij}$ 
and $t_k$, and does not depend at all on the ambient differential
operators.

Our task becomes to translate the situation from
Subsection~\ref{ssec:bFunctionDefMonomial} to the new setting with
monomial ideals in semigroup rings, rather than monomial ideals in
polynomial rings.

Let $c \in \ZZ^r$ whose coordinates sum to $1$, and for $\beta
=(\beta_1,\dots,\beta_r) \in (\ZZ^d)^r$, set 
$\ell_\beta(s_1,\dots,s_r)
  =s_1\beta_1+\cdots+s_r\beta_r$.

In order to find the Bernstein--Sato polynomial, we need to apply an operator
in $D_A[s_1,\dots,s_r]$ to $\prod_{i \in \nsupp(c)} \binom{s_1}{-c_i} \cdot \prod_{i=1}^r
  (y^{\beta_i})^{s_i+c_i}$ in such a way that the outcome is a
  multiple of $y^{\sum_{i=1}^rs_i \beta_i} =
  y^{\ell_\beta(s_1,\dots,s_r)}$.

If we apply the operator from~\eqref{eqn:differentialGenerator} with
$u=\ell_\beta(c)$ to  $\prod_{i
  \in \nsupp(c)} \binom{s_1}{-c_i} \cdot \prod_{i=1}^r
  (y^{\beta_i})^{s_i+c_i}$ we obtain
\[
\prod_{i \in \nsupp(c)} \binom{s_1}{-c_i}
\prod_{F_{\sigma}(\ell(c))>0} \prod_{j=0}^{F_{\sigma}(\ell(c))-1} (
F_\sigma(\ell_\beta(s_1,\dots,s_r) + \ell_\beta(c))-
j ) 
\cdot
 y^{\ell(s_1,\dots,s_r)} .
\]

\begin{proposition}
\label{prop:simplifiedbFunctionSemigroupRing}
We use the notation introduced above. 
The Bernstein--Sato polynomial $b_f(s)$ 
for $f =(y^{\beta_1},\dots,y^{\beta_r})$ considered as a monomial
ideal in $\kk[\NN A]$ is the monic polynomial of smallest degree such
that $b_f(s_1+\cdots+s_r)$ belongs to the ideal generated by 
\begin{equation}
\label{eqn:bFunctionGenerator}
\prod_{i \in \nsupp(c)} \binom{s_1}{-c_i} 
\prod_{F_{\sigma}(\ell_\beta(c))>0}
\binom{F_\sigma\big(\ell_\beta(s_1,\dots,s_r)+\ell_\beta(c)\big)}{F_{\sigma}(\ell_\beta(c))} 
\end{equation}
where $c \in \ZZ^r$ has coordinate sum $1$.
\end{proposition}

\begin{proof}
It remains to be shown that, for any other operator (in
$D_A[s_1,\dots,s_r]$) which applied to $\prod_{i
  \in \nsupp(c)} \binom{s_1}{-c_i} \cdot \prod_{i=1}^r
  (y^{\beta_i})^{s_i+c_i}$ yields a multiple of
  $y^{\ell(s_1,\dots,s_r)}$, this multiple belongs to the ideal
  generated by~\eqref{eqn:bFunctionGenerator}.

This follows from the description of $D_A$ given by~\eqref{eqn:DA}
and~\eqref{eqn:differentialGenerator}. 
\end{proof}

Recall that $\mathscr{F}$ is the collection of facets of the cone $C$
of $\NN A$, and that the map $F: \RR^d \to \RR^{\mathscr{F}} $ defined
by $F(p) = (F_{\sigma}(p))_{\sigma \in \mathscr{F}}$  from
Notation~\ref{not:intro}.\ref{not:intro:cone} 
induces an inclusion $\kk[\NN A] \rightarrow \kk[\NN^\sF]$.

\begin{proposition}
\label{prop:bFunctionRelation}
Denote by $F(f)$ the sequence of monomials $(x^{F(\beta_1)},\dots,x^{F(\beta_r)})$ in $\kk[\NN^\sF]$.
The Bernstein--Sato polynomial $b_{F(f)}(s)$ of $F(f)$ in the  
polynomial ring $\kk[\NN^\sF]$ coincides with the Bernstein--Sato polynomial
 $b_f(s)$ of $f=(x^{\beta_1},\dots,x^{\beta_r})$ in $\kk[\NN A]$. In
 particular, the latter polynomial $b_f(s)$ is nonzero, 
and its roots can be computed 
using the 
combinatorial description~\cite[Theorem~1.1]{BMS06b} applied to
$\<x^{F(\beta_1)},\dots,x^{F(\beta_r)}\> \subseteq \kk[\NN^\sF]$.
\end{proposition}

\begin{proof}
First we observe that the minimal monomial generators of
$\<x^{F(\beta_1)},\dots,x^{F(\beta_r)}\> \subseteq \kk[\NN^\sF]$
are the monomials $x^{F(\beta_1)},\dots,x^{F(\beta_r)}$.
To see this, note that $x^{F(\beta_i)}$ divides $x^{F(\beta_j)}$ is
equivalent to $F_\sigma(\beta_i) \leq F_\sigma(\beta_j)$ for all
facets $\sigma$ of the cone $C$. This implies that $\beta_j-\beta_i$
lies in the cone $C$. Since $\beta_j-\beta_i \in \ZZ^r$ and $\NN A$ is
normal, we see that $\beta_j-\beta_i \in \NN A$, and therefore
$y^{\beta_i}$ divides $y^{\beta_j}$ in $\kk[\NN A]$.

Let $\alpha_i = F(\beta_i)$ for $i=1,\dots,r$, and
$\ell_\alpha(s_1,\dots,s_r) = s_1\alpha_1+\cdots+s_r\alpha_r$, and
denote by $\ell_i(s_1,\dots,s_r)$ the $i$th coordinate of the vector
$\ell_\alpha(s_1,\dots,s_r)$. 
By~\cite[Proposition~4.2]{BMS06} 
the Bernstein--Sato polynomial of
$\<x^{F(\beta_1)},\dots,x^{F(\beta_r)}\>$ is the monic polynomial $p$
of smallest degree such that $p(s_1+\dots+s_r)$ lies in the ideal
generated by
\begin{equation}
\label{eqn:generatorsTransformed}
\prod_{j \in \nsupp(c)} \binom{s_j}{-c_j} 
\prod_{i \in  \psupp(\ell_{\alpha}(c))} \binom{\ell_i(s_1,\dots,s_r)+\ell_i(c)}{\ell_i(c)}
\end{equation}
for $c\in \ZZ^r$ with coordinate sum $1$.
But note that by construction $\ell_\alpha = F \circ \ell_\beta$, 
so that the generators~\eqref{eqn:generatorsTransformed}
coincide exactly with the generators~\eqref{eqn:bFunctionGenerator}.
As a reality check, note that both sets of generators belong to the
ring $\kk[s_1,\dots,s_r]$, which does not depend on the ambient rings
of the monomial ideals involved.
\end{proof}
%%%%%%%%%%%%%%%%%%%%%%%%%%%%%%%%%%%%%%%%%%%%%%%%%%%%%%%%%%%%%%%%%%%%%%%%%
\subsection{The general case}
\label{ss:toricBernsteinSatoGeneral}
%%%%%%%%%%%%%%%%%%%%%%%%%%%%%%%%%%%%%%%%%%%%%%%%%%%%%%%%%%%%%

Again, recall that we have a linear map 
$F:\RR^d \to \RR^{\mathscr{F}}$ with $F(\NN A) \subseteq
\NN^{\mathscr{F}}$. We construct a ring homomorphism 
\begin{equation}
\label{eqn:ringMap}
\kk[\NN A]\to \kk[\NN^{\mathscr{F}}]  \qquad \text{ via } \qquad
\sum_{j=1}^k \lambda_{j} y^{\beta_j} \mapsto \sum_{j=1}^k \lambda_j x^{F(\beta_j)},
\end{equation}
where $\beta_1,\dots,\beta_j \in \NN A$. That this is a homomorphism
follows from linearity of $F$, as $x^{F(\beta+\beta')}=x^{F(\beta)+F(\beta')}=x^{F(\beta)}x^{F(\beta')}$.
We abuse notation and denote the homomorphism~\eqref{eqn:ringMap} by
$F$.

If $I \subset \kk[\NN A]$ is an ideal, then its image $F(I)$ is not an ideal in
$\kk[\NN^{\mathscr{F}}]$; we denote by $\<F(I)\>$ the ideal in
$\kk[\NN^{\mathscr{F}}]$ generated by $F(I)$.

\begin{lemma}
Let $I \subset \kk[\NN A]$ be an ideal generated by $g_1,\dots,g_r \in
\kk[\NN A]$. Then $F(g_1),\dots,F(g_r)$ generate $\<F(I)\>$.
\end{lemma}

\begin{proof}
Since $g_1,\dots,g_r$ generate $I$, and $F$ is a ring homomorphism,
any element of $F(I)$ is obtained as a combination of
$F(g_1),\dots,F(g_r)$ with coefficients in
$\kk[\NN^{\mathscr{F}}]$. This implies that the polynomials
$F(g_1),\dots,F(g_r)$ generate $\<F(I)\>$.
\end{proof}

We apply Definition~\ref{def:bFunction} 
to a sequence of polynomials in $\kk[\NN A]$, using the ring of
differential operators $D_A$ instead of $D$.
The following is the main result of this section.

\begin{theorem}\label{thm:ToricbFunctionGeneral}
Let $I \subset \kk[\NN A]$ be an ideal, and let $g_1,\dots,g_r$ be
generators for $I$. 
The Bernstein--Sato polynomial of $g=(g_1,\dots,g_r)$ in $\kk[\NN A]$ equals the
Bernstein--Sato polynomial of $f=(f_1,\dots,f_r)=F(g) = (F(g_1),\dots,F(g_r))$ in
$\kk[\NN^{\mathscr{F}}]$.
Consequently, this Bernstein--Sato polynomial is nonzero, and depends only on $I$,
not on the particular set of generators chosen.
\end{theorem}

\begin{proof}
We know that the Bernstein--Sato polynomial $b_g(s)$ is the monic
polynomial of smallest degree such that, for $s=s_1+\cdots+s_r$, 
$b_g(s)\prod_{i=1}^r g_i^{s_i}$ belongs to the $D_A[s_1,\dots,s_r]$-submodule generated by 
\[
\prod_{i\in \nsupp(c)} \binom{s_i}{-c_i} \cdot \prod_{i=1}^r
g_i^{s_i+c_i}, \qquad \text{for}\;\; (c_1,\dots,c_r)\in \ZZ^r
\;\text{with} \;\; \sum_{i=1}^r c_i=1,
\]
while $b_f(s)$ is is the monic
polynomial of smallest degree such that, 
$b_f(s)\prod_{i=1}^r f_i^{s_i}$ belongs to the $D[s_1,\dots,s_r]$-submodule generated by 
\[
\prod_{i\in \nsupp(c)} \binom{s_i}{-c_i} \cdot \prod_{i=1}^r
f_i^{s_i+c_i}, \qquad \text{for}\;\; (c_1,\dots,c_r)\in \ZZ^r
\;\text{with} \;\; \sum_{i=1}^r c_i=1.
\]
Here $D$ is the ring of differential operators on
$\kk[\NN^{\mathscr{F}}]$, and $D_A$ is the ring of differential
operators on $\kk[\NN A]$.

Fix $c \in \ZZ^r$ such that $\sum c_i=1$, and let $P \in
D[s_1,\dots,s_r]$ such that 
$P \big[ \prod_{i\in \nsupp(c)} \binom{s_i}{-c_i} \cdot \prod_{i=1}^r
f_i^{s_i+c_i} \big]$ is a polynomial in $s_1,\dots,s_r$ times
$\prod_{i=1}^r f_i^{s_i}$. 
Applying the description given by~\eqref{eqn:DA}
and~\eqref{eqn:differentialGenerator} to $D$ (instead of $D_A$), we see that 
we
can write $P$ as a finite sum 
\[
P = \sum_{u \in \ZZ^{\mathscr{F}}} q_u(s_1,\dots,s_r) x^u
\bigg[ \prod_{u_{\sigma}<0} \prod_{j=0}^{-u_{\sigma}-1}\big(
(\theta_x)_{\sigma} - j \big) \bigg] p_u(\theta_x), 
\]
where the $p_u$ are
polynomials in $|\mathscr{F}|$ indeterminates with coefficients in
$\kk$. 

Note that, by construction, the element
$\prod_{i\in \nsupp(c)} \binom{s_i}{-c_i} \cdot \prod_{i=1}^r
f_i^{s_i+c_i}$ is $F(\NN A)$-graded. Since $P \big[ \prod_{i\in
  \nsupp(c)} \binom{s_i}{-c_i} \cdot \prod_{i=1}^r 
f_i^{s_i+c_i} \big]$ is a multiple of $\prod_{i=1}^r f_i^{s_i}$, we
may assume that the
operator $P$ is $F(\NN A)$-graded as well. In other words, 
we may assume that $u \in F(\ZZ^d)$ if $q_u p_u \neq 0$.

Thus we rewrite
\[
P = \sum_{u=F(v) \in F(\ZZ^d)} q_u(s_1,\dots,s_r) x^u
\bigg[ \prod_{u_{\sigma}<0} \prod_{j=0}^{-u_{\sigma}-1}\big(
(\theta_x)_{\sigma} - j \big) \bigg] p_u(\theta_x), 
\]
and if we denote
\[
\hat{P} = \sum_{v \in \ZZ^d} q_{F(v)}(s_1,\dots,s_r) y^v 
\bigg[ \prod_{F_{\sigma}(v)<0} \prod_{j=0}^{-F_{\sigma}(v)-1}\big(
F_\sigma(\theta_y) - j \big) \bigg]
p_{F(v)}((F_{\sigma}(\theta_y))_{\sigma\in \mathscr{F}}), 
\]
then $\hat{P}$ is an element of $D_A[s_1,\dots,s_r]$, and
$\hat{P}$ applied to $\prod_{i\in \nsupp(c)} \binom{s_i}{-c_i} \cdot \prod_{i=1}^r
g_i^{s_i+c_i}$ is a polynomial of $s_1,\dots,s_r$ times
$\prod_{i=1}^r g_i^{s_i}$, and this is the same polynomial that we
obtain when we apply $P$ to 
$\prod_{i\in \nsupp(c)} \binom{s_i}{-c_i} \cdot \prod_{i=1}^r
f_i^{s_i+c_i}$ (and divide by $\prod f_i^{s_i}$).

Conversely, if we apply an element of $D_A[s_1,\dots,s_r]$ to
$\prod_{i\in \nsupp(c)} \binom{s_i}{-c_i} \cdot \prod_{i=1}^r g_i^{s_i+c_i}$ 
and obtain a polynomial in $s_1,\dots,s_r$ times $\prod_{i=1}^r
g_i^{s_i}$, then we can obtain an element of $D[s_1,\dots,s_r]$ which
applied to $\prod_{i\in \nsupp(c)} \binom{s_i}{-c_i} \cdot
\prod_{i=1}^r f_i^{s_i+c_i}$ gives exactly the same polynomial in
$s_1,\dots,s_r$ times $\prod_{i=1}^r f_i^{s_i}$. To do this, we write
the element of $D_A[s_1,\dots,s_r]$ as a sum of terms of the form
\[
q_v(s_1,\dots,s_r)y^v \bigg[\prod_{F_\sigma(v)<0}
\prod_{j=0}^{-F_\sigma(v)-1}(F_\sigma(\theta_y)-j)\bigg]
p_v(\theta_y),
\]
then express (non uniquely) the polynomial $p_v$ as sums of powers of
the linear forms $F_\sigma(\theta_y)$, and then apply $F$ in the
obvious way. Nonuniqueness comes because we need to choose $d$
  linearly independent forms $F_\sigma$ in order to get $\kk$-algebra
  generators of $\kk[\theta_y]$; however, since the image of $F$ as a
  linear map is $d$-dimensional, the choice does not affect the image.
\end{proof}

%%%%%%%%%%%%%%%%%%%%%%%%%%%%%%%%%%%%%%%%%%%%%%%%%%%%%%%%%%%%%%%%%%%%%%%%%
\section{Multiplier Ideals on Normal Toric Varieties}
\label{s:MultiplierIdealNormalToricVariety}
%%%%%%%%%%%%%%%%%%%%%%%%%%%%%%%%%%%%%%%%%%%%%%%%%%%%%%%%%%%%%
We recall basic definitions of multiplier ideals that can be found in
Lazarsfeld's text, Positivity in Algebraic Geometry II~\cite{Laz04b}.

Let $X$ be a smooth variety over an algebraically closed field $\kk$
of characteristic $0$. Let $\cI \subseteq \cO_X$ be an ideal sheaf,
and $\alpha>0$ a rational number. Fix a log resolution $\mu: X'
\rightarrow X$ of $\cI$ with $\cI\cdot \cO_{X'} = \cO_{X'}(-E)$. The
multiplier ideal of the pair $(X, \alpha \cI)$ is defined as  
\[ 
\cJ(X,\alpha \cI) = \mu_* \cO_{X'}(K_{X'/X}- \lfloor \alpha \cdot E
\rfloor),
\]
where $K_{X'/X}=K_{X'}- \mu^*K_X$ is the relative canonical divisor of
$\mu$, which is an effective divisor supported on the exceptional
locus of $\mu$ whose local equation is given by the determinant of the
derivative $d \mu$. 
The definition does not depend on the choice of log resolution. One of
the important features of multiplier ideals is that $\cJ(X,\alpha
\cI)$ measures the singularity of the pair $(X,\alpha \cI)$: a smaller
multiplier ideal corresponds to a worse singularity. 

Notice that $\cJ(X,\alpha \cI)$ becomes smaller as $\alpha$ increases.
The jumping coefficients of $(X,\alpha \cI)$ are the positive real
numbers $0<\alpha_1<\alpha_2<\dots$ such that $\cJ(X,\alpha_j
\cI)=\cJ(X,\alpha \cI)\ne 
\cJ(X,\alpha_{j+1} \cI)$ for $\alpha_j\le \alpha <\alpha_{j+1} (j\ge 0)$ where $\alpha_0=0$. 
When $X$ is affine and $I$ is the ideal in the coordinate ring
$\kk[X]$ corresponding the the sheaf $\cI$, Budur, Musta\c{t}\u{a} and
Saito proved that jumping coefficients of $(X,\alpha I)$ in
$[\alpha_f,\alpha_f+1)$ are roots of $b_f(-s)$, where
$f=(f_1,\dots,f_r)$ is a set of generators for $I$, $b_f(s)$ is the
Bernstein--Sato polynomial of $I$, and $\alpha_f$ is the smallest root
of $b_f(-s)$. One of our goals is to generalize this correspondence
between $b$-function roots and jumping coefficients
to monomial ideals on affine normal toric varieties.

In the special case that $I$ is a monomial ideal in the polynomial
ring $\kk[x_1,\dots,x_m]$, Howald gave the following combinatorial
formula for multiplier ideal of the pair $(\AA^m, \alpha I)$, 
\[
\cJ(\AA^m, \alpha I) = \< x^v \mid v+e \in \relint \left(\alpha
  P_I\right) \>,
\]
where $P_I$ is the Newton polyhedron of $I$ and $e$ is the vector
$(1,1,\dots,1)$ in $\NN^m$. In particular, a rational number
$\alpha>0$ is a jumping coefficient of $(\AA^m, \alpha I)$ if and only
if the boundary of $-e + \relint \left( \alpha P_I\right)$ contains a
lattice point in $\NN^m$.

The notion of multiplier ideal can be generalized to the case that
$\cI$ is an ideal sheaf on a normal variety $X$ over $\kk$ as
follows. 
Let $\Delta$ be an effective $\QQ$-divisor such that $K_X+\Delta$ is
$\QQ$-Cartier. Such $\Delta$ is called a boundary divisor. Let $\mu :
X' \rightarrow X$ be a log resolution of the triple $(X,\Delta, \cI)$
and $\alpha >0$ be a rational number. Suppose that $\cI\cdot \cO_{X'}=
\cO_{X'}(-E)$. 
Then one can define the multiplier ideal $\cJ(X,\Delta, \alpha \cI)$
associated to the triple $(X,\Delta, \alpha \cI)$ as 
\[
\cJ(X,\Delta, \alpha \cI) = \mu_* \cO_{X'} \left( K_{X'} - \lfloor
  \mu^* (K_X +\Delta) +\alpha E \rfloor \right).
\]
Again, this definition is not dependent on the choice of $\mu$.
De Fernex and Hacon~\cite{dFH09} have given a definition of
$\cJ(X,\alpha Z)$ for non-$\QQ$-Gorenstein $X$ without using the
boundary divisor $\Delta$. They showed 
that there exists a boundary divisor $\Delta$ such that the multiplier
ideal $\cJ(X,\Delta, \alpha Z)$ coincides with their multiplier ideal
$\cJ(X,\alpha Z)$ and that $\cJ(X,\alpha Z)$ is the unique maximal
element of the set  
\[
\{\cJ(X,\Delta, \alpha Z) \mid \Delta \text{ is a boundary divisor}\}.
\]

\subsection{A New expression for multiplier ideals on toric varieties}
\label{ssec:NewExpressionofToricMultiplierIdeal}
Let us explain how one can use the map $F$ to understand the
multiplier ideals of De Fernex and Hacon in the case of monomial
ideals on normal toric varieties. 

Let $\kk[\NN A] \subset \kk[y_1^\pm, \dots,y_d^\pm]$ be a normal
semigroup ring in our setting and let $X= \Spec(\kk[\NN A])$. Each
facet $\sigma \in \sF $ corresponds to a torus invariant prime Weil
divisor $D_\sigma$ on $X$.   
The canonical class of $X$ is represented by the torus invariant
canonical divisor $K_X= -\sum_{\sigma \in \sF} D_\sigma$. A boundary
divisor $\Delta$ is an effective $\QQ$-divisor such that $K_X+\Delta$
is $\QQ$-Cartier, which means there exist $l \in \ZZ$ and $u\in \ZZ^d$
such that $l(K_X +\Delta)=\div (y^u)$. 
Denote $w_\Delta = {u \over l}$. Notice that 
$$
\Delta = \sum_{\sigma \in \sF} (1+F_\sigma(w_\Delta))D_\sigma,
$$
 so the effectivity of $\Delta$ is equivalent to the condition that
 $F_\sigma(w_\Delta) \ge -1$ for all $\sigma \in \sF$. 
 Conversely, each $w \in \QQ^d$ gives rise to a boundary divisor
 $\Delta_w$ on $X$ by the same formula 
\[
\Delta_w = \sum_{\sigma \in \sF} (1+F_\sigma(w))D_\sigma.
\]
 
Let $I$ be a monomial ideal in $\kk[\NN A]$ and let $\alpha >0$ be a rational number.
Blickle~\cite{Bli04} showed that the multiplier ideal 
$\cJ(X,\Delta, \alpha I) = \< y^v \in \kk[\NN A]  \mid v-w_\Delta \in
\relint(\alpha P_I) \>,$ generalizing Howald's description of multiplier ideals~\cite{How01}
in the case of monomial ideals in polynomial rings. Be cautioned about
the mistake of the sign of $w_\Delta$ in Blickle's original
description. 
For $w \in\QQ^d$, denote 
\begin{equation}
\label{eqn:MultiplierIdealExponents}
 \Omega_{w, \alpha I} = \left[w +\relint (\alpha P_I) \right] \text{ and }
\Omega_{\alpha I}= \bigcup_{w \in \QQ^d: F_\sigma(w) \ge -1 \,  \forall \sigma \in \sF} \Omega_{w, \alpha I}.
\end{equation}
Then the multiplier ideal of De Fernex and Hacon is
\begin{equation}
\label{eqn:DF-HMultiplierIdeal}
\cJ(X,\alpha I) = \< y^v \in \kk[\NN A] \mid v \in \Omega_{\alpha I}\>.
\end{equation}

We claim that $\cJ(X, \alpha I)$ can be computed using an analog of
Howald's formula. Recall that the linear map $F:\RR^d \rightarrow
\RR^\sF$ induces ring homomorphisms 
\[ 
\kk[\NN A] \xrightarrow{\sim} \kk[F(\NN A)] \to \kk[\NN^\sF].
\]
By abusing notation, denote 
\begin{equation}
\label{eqn:F(I)}
F(I) = \kk[F(\NN A)]\cdot \< x^{F(v)} \mid y^v \in I \>
\end{equation}
the ideal in $\kk[F(\NN A)]$ obtained from the semigroup isomorphism
$F: \NN A \xrightarrow{\sim} F(\NN A)$.  
The monomial ideal in $\kk[\NN^\sF]$ generated by $F(I)$ is
denoted by 
\begin{equation}
\label{eqn:<F(I)>}
J=\kk[\NN^\sF]\cdot \< F(I)\>.
\end{equation}

Let $e$ be the element in $\RR^\sF$ such that $e_\sigma=1$ for all
$\sigma  \in \sF$. Notice that $e$ may not be in $F(\NN A)$ and that
$e \in F(\NN A)$ if and only if $X$ is Gorenstein. Even if one extends
to rational coefficients, the element $e$ may not be in $F(\QQ \otimes
\NN A) =F(\QQ^d)$. The condition $e \in F(\QQ^d)$ holds exactly when
$X$ is $\QQ$-Gorenstein, namely  $K_X$ is a $\QQ$-Cartier
divisor. This is the case when $X$ is $\QQ$-factorial (or,
equivalently, the semigroup $\NN A$ is simplicial).  

In the case where $e \in F(\QQ^d)$, it is clear that $\Omega_{\alpha
  I}=\Omega_{F^{-1}(-e),\alpha I},$ so $w=F^{-1}(-e)$ corresponds to
the $\Delta$ of De Fernex and Hacon. In fact, the divisor
$\Delta_{F^{-1}(-e)} = 0$ coincides with the canonical choice of
boundary divisor for the pair $(X,\alpha I)$ in the $\QQ$-Gorenstein
case. Moreover, we have the following analog of Howald's formula for
$\cJ(X,\alpha I)$. 
\begin{proposition}
\label{prop:ToricHowaldFormulaQGor} 
If $X$ is $\QQ$-Gorenstein, then
$$\cJ(X,\alpha I)  =\< y^v\in \kk[\NN A]  \mid F(v)+e \in \relint (\alpha P_{F(I)}) \>.$$
\end{proposition}
\begin{proof}
\begin{equation}
\label{eqn:ToricHowaldFormula}
\begin{aligned}
\cJ(X,\alpha I) &= \<y^v\in \kk[\NN A]  \mid v \in \Omega_{\alpha I} \>\\
&=\< y^v \in \kk[\NN A] \mid v \in \Omega_{F^{-1}(-e), \alpha I} \>\\
&=
\< y^v\in \kk[\NN A]  \mid F(v)+e \in \relint (\alpha F(P_I)) \>\\
&=\< y^v\in \kk[\NN A]  \mid F(v)+e \in \relint (\alpha P_{F(I)}) \>.
\end{aligned}
\end{equation}
Since $P_I =\cvx \left[\bigcup_{v: y^v \in I} \left(v+C\right)\right]$
is the convex hull in $\RR^d$ of the set $\bigcup_{y^v \in I}
\left(v+C\right)$, we see that $F(P_I) =\cvx \left[\bigcup_{v: y^v \in
    I} \left(F(v)+F(C)\right)\right] $ is exactly the Newton
polyhedron $P_{F(I)}$ in $F(\RR^d)$ of the ideal $F(I)$ in the
semigroup ring $\kk[ F(\NN A)]$. The last expression
in~\eqref{eqn:ToricHowaldFormula} is analogous to Howald's formula.   
\end{proof}

{\bf Notations for Newton Polytopes.} For a monomial ideal $I$ in
$\kk[\NN A] \subset \kk[y_1^{\pm 1},\dots,y_d^{\pm 1}]$, denote $P_I$
the Newton polyhedron of  $I$, which by definition is the convex hull
of  
$ \bigcup_{v:y^v \in I} \left( v+\NN A\right)$ in $\RR^d$. The
relative interior of $P_I$ is denoted by $\relint P_I$.

Recall from~\eqref{eqn:F(I)} that $F(I)$ denotes a monomial ideal in
$\kk[\NN A]$. 
The Newton polyhedron of $F(I)$ is denoted by $P_{F(I)}$; this is the
convex hull of $\{F(v) \mid y^v \in I\}$ in $F(\RR^d)$. 
We have 
\[
P_{F(I)} =\cvx \left[ \bigcup_{v: y^v \in I} (F(v)+F(\NN A))\right].
\]
Recall also from~\eqref{eqn:<F(I)>} that $J=\kk[\NN^\sF]\cdot\<F(I)\>$ denotes the
monomial ideal in $\kk[\NN^{\sF}]$ generated by $F(I)$.
The Newton polyhedron of $J$ in $\RR^\sF$ is 
\[
P_J = \cvx \left[ \bigcup_{v: y^v \in I}
  \left(F(v)+\NN^\sF\right)\right].
\]

Note that if $y^{\beta_1}, \dots, y^{\beta_r}$ is the set of minimal
monomial generators of $I$, then $\{\beta_1, \dots, \beta_r\}$
(respectively, $\{F(\beta_1), \dots, F(\beta_r)\}$) is exactly the set
of vertices of the Newton polyhedron $P_I$ (respectively, $P_{F(I)}$
or $P_J$).

For general $X$, we give a description of $\cJ(X,\alpha I)$ using the
Newton polyhedron of the ideal $J=\kk[\NN^\sF]\cdot\<F(I)\>$. We first
need to compare the relative interiors of $\alpha P_{F(I)}$ and
$\alpha P_J$. 
\begin{lemma}
\label{lem:P_F(I)&P_J} For general $\NN A$, the relative interiors of
$P_{F(I)}$ and $P_J$ satisfy 
$$\relint \left( \alpha P_{F(I)} \right) = F(C) \cap \relint \left( \alpha P_{J} \right).$$
\end{lemma}
\begin{proof}
A point $u$ lies in $\relint( P_{F(I)})$ if and only if $u-v \in
\relint (F(C))$ for some $v$ on a bounded face of
$P_{F(I)}$. Similarly, a point $u$ lies in $\relint( P_J)$ if and only
if $u-v \in \relint \RR^\sF_{\ge 0}$ for some $v$ on a bounded face of
$P_J$.
Since $F(I)$ and $J$ have the same minimal monomial generators, the
bounded faces of $P_{F(I)}$ and $P_J$ coincide. Therefore, it suffices
to show that 
\[
\relint( F(C)) = F(C) \cap \relint (\RR_{\ge 0 }^\sF).
\]
But for $p\in C$, $F(p) \in \relint F(C)$ if and only if $p$ is not
contained in any facet $\sigma \in \sF$, which means exactly
$F_\sigma(p) >0$. 
\end{proof}

\begin{proposition}
\label{prop:ToricHowaldFormulaGeneral}
The multiplier ideal
$$\cJ(X,\alpha I) = \< y^v \in \kk[\NN A] \mid F(v)+e  \in \relint(\alpha P_{J})\>.$$
\end{proposition}
\begin{proof}
If $X$ is $\QQ$-Gorenstein, the statement follows from
Proposition~\ref{prop:ToricHowaldFormulaQGor} and
Lemma~\ref{lem:P_F(I)&P_J}.  

In general, by \eqref{eqn:MultiplierIdealExponents} and
\eqref{eqn:DF-HMultiplierIdeal} it suffices to show that  
\[
F(\Omega_{\alpha I} \cap \NN A)  = \left( -e + \relint \left( \alpha
    P_J\right)\right) \cap F(\NN A).
\]

For the containment $F(\Omega_{\alpha I} \cap \NN A)  \subseteq \left( -e + \relint \left( \alpha P_J\right)\right) \cap F(\NN A)$,
it is enough to verify that 
$$
F\left(w+\relint (\alpha P_I)\right)
\subseteq \left[ -e +\relint \left( \alpha P_{J}\right)\right]
$$ 
for any $w\in \QQ^d$ satisfying $F_\sigma( w) \ge -1$ for all $\sigma \in
\sF$. Any such $w$ satisfies $F(w)+e \in \RR^\sF_{\ge 0}$, so by
Lemma~\ref{lem:P_F(I)&P_J} 
\[
\begin{aligned}
F(w+\relint(\alpha P_I))=&\left[F(w)+\relint \left(\alpha F(P_I)\right)\right] \\
=&\left[F(w)+\relint \left(\alpha P_{F(I)}\right)\right] \\
\subseteq &\left[F(w)+\relint \left(\alpha P_{J}\right)\right] \\
 =  &\left[-e+(F(w)+e)+\relint \left(\alpha P_{J}\right)\right]\\
 \subseteq  &\left[-e+\RR^\sF_{\ge 0}+\relint \left(\alpha P_{J}\right)\right]\\
  \subseteq  &\left[-e+\relint \left(\alpha P_{J}\right)\right].
 \end{aligned}
\]

For the other containment $F(\Omega_{\alpha I} \cap \NN A)  \supseteq
\left( -e + \relint \left( \alpha P_J\right)\right) \cap F(\NN A)$,
let $v \in \NN A$ be such that $F(v) + e \in \relint (\alpha P_J)$.  
Then there exists $u$ lying on a bounded face of $\alpha P_I$ such that 
\[
F(v)+e -F(u) \in \relint(\RR_{\ge 0}^\sF)
\]
In particular, $F_\sigma(v-u) >-1$ for all $\sigma \in \sF$. Take any
$p \in \relint C$ and $\epsilon >0$ small enough so that 
\[
w:= v-u-\epsilon p \in \QQ^d \text{ and } F_\sigma(v-u-\epsilon p)
>-1.
\]
Then $F(v)-F(u)-F(w) = F(\epsilon p) \in \relint C$ where $F(u)$ lies
on a bounded face of $\alpha P_{F(I)}$. 
Therefore, we have $F(v) - F(w) \in \relint \left(\alpha  P_{F(I)}
\right)$, and hence $F(v) \in F(\Omega_{ w, \alpha I} \cap \NN A)
\subseteq F(\Omega_{\alpha I} \cap \NN A)$ as desired.   
\end{proof}

\subsection{Roots--Jumping coefficients correspondence on toric varieties}
\label{ssec:ToricRootsJumpingCoefficientsCorrespondence}
Now, we combine the previous observations to establish the following theorem.
\begin{theorem}
\label{thm:ToricRootsJumpingCoefficientsCorrespondence}
Let $\NN A$ be a normal semigroup, $X = \Spec(\kk[\NN A])$ its
associated affine toric variety and let $I$ be a monomial ideal on
$X$. Suppose $\alpha_I$ is the smallest root of $b_I(-s)$ where
$b_I(s)$ is the Bernstein--Sato polynomial of $I$ in $\kk[\NN
A]$. Then any jumping coefficients of the pair $(X, I)$ in $[\alpha_I,
\alpha_I+1)$ are roots of $b_I(-s)$. Moreover, the number $\alpha_I$
is the smallest jumping coefficient (i.e. the log-canonical threshold)
of $(X,I)$. 
\end{theorem}
\begin{proof}
By Proposition~\ref{prop:bFunctionRelation}, the Berstein--Sato
polynomial $b_I(s)$ of $I$ coincides with the Bernstein--Sato
polynomial of the monomial ideal $J=\kk[\NN^\sF]\cdot\<F(I)\>$. Thus
by~\cite[Theorem~2]{BMS06}, the jumping coefficients of the pair $(
\AA^\sF, J)$ in $[\alpha_I, \alpha_I+1)$ are roots of $b_I(-s)$. By
Howald's formula, a number $\alpha$ is a jumping coefficient of $(
\AA^\sF,  J)$ exactly when the boundary of $\left(-e+\alpha P_{J}
\right)$ contains a lattice point in $\NN^\sF$. Also, according to
Proposition~\ref{prop:ToricHowaldFormulaGeneral} a number $\alpha$ is
a jumping coefficient of $( X, I)$ exactly when the boundary of
$\left(-e+\alpha P_{J} \right)$ contains a lattice point in $F(\NN
A)$. Therefore, the jumping coefficients of $(X, I)$ are jumping
coefficients of $(\AA^\sF, J)$, and the first statement of this theorem
follows.

To prove $\alpha_I$ is the log-canonical threshold of $(X,I)$, it
suffices to show that the boundary of $ (-e + \alpha_I P_J)$
intersects $F(\NN A)$. Since $\alpha_I$ is the log-canonical threshold
of $(\AA^\sF,J)$,  we have 
\[
1 \in \kk[\NN^\sF]=\mathcal J( \AA^\sF, \alpha_I J) \text{, and hence }
 F(0)=0 \in (-e + \alpha_I P_J) \cap F(\NN A).
\]
\end{proof}

\begin{example}
Let $A=\begin{pmatrix} 1&1&1&1\\0&1&2&3\end{pmatrix}$ and $I =
\<y_1y_2,y_1y_2^2 \>$ a monomial ideal in $\kk[\NN A]$. In this case,
the semigroup ring $\kk[\NN A]$ is simplicial and hence
$\QQ$-Gorenstein. The linear mapping $F: \RR^2 \rightarrow \RR^2$ is
represented by the matrix
$\begin{pmatrix}3&-1\\0&\phantom{-}1\end{pmatrix}$. The subsemigroup
$\kk[F(\NN A)]$ of $\kk[x_1,x_2]$ is generated by
$x_1^3,x_1^2x_2,x_1x_2^2,x_2^3$ and the monomial ideal $J=
\kk[x_1,x_2]\<x_1^2x_2,x_1x_2^2\>$. Using the method of Budur,
Musta\c{t}\u{a} and Saito
as discussed in subsection~\ref{ssec:bFunctionDefMonomial}, one can
compute the Bernstein--Sato polynomial
$b_{J}(s)=(s+1)^2(3s+2)(3s+4)$. (We point out that $b$-function algorithms
have been developed and implemented, see~\cite{BL10}.)
Moreover, by Proposition~\ref{prop:bFunctionRelation} we have
$b_I(s)=(s+1)^2(3s+2)(3s+4)$ as well. On the other hand, using
Howald's formula, one finds that ${2\over 3},1,{4\over 3}$ are jumping
coefficients of $(\AA^2, J)$, but only ${2\over 3},1$ are jumping
coefficients of $(X,  I)$ according to
Proposition~\ref{prop:ToricHowaldFormulaGeneral}. 
\end{example}

\section*{Acknowledgements}

The first author thanks Texas A\&M University for the hospitality he
enjoyed during his visit in January 2016, when this work was
initiated.

%%%%%%%%%%%%%%%%%%%%%%%%%%%%%%%%%%%%%%%%%%%%%%%%%%%%%%%%%%%%%%%%%%%%%%%%%
\bibliographystyle{amsalpha}
\bibliography{bfunction}
%%%%%%%%%%%%%%%%%%%%%%%%%%%%%%%%%%%%%%%%%%%%%%%%%%%%%%%%%%%%%%%%%%%%%%%%%

%%%%%%%%%%%%%%%%%%%%%%%%%%%%%%%%%%%%%%%%%%%%%%%%%%%%%%%%%%%%%%%%%%%%%%%%%
\end{document}